\documentclass[preprint, 12pt]{elsarticle}
\usepackage{amssymb, amsthm, amsmath}

\theoremstyle{plain}
\newtheorem{theorem}{Theorem}
\newtheorem{coro}[theorem]{Corollary}
\newtheorem{lemma}[theorem]{Lemma}

\theoremstyle{definition}
\newtheorem{definition}[theorem]{Definition}
\newtheorem{example}[theorem]{Example}
\newtheorem{rmk}[theorem]{Remark}

\numberwithin{theorem}{section}

\newcommand{\ZZ}{\mathbb{Z}} 
\newcommand{\NN}{\mathbb{N}} 

\newcommand{\CC}{\mathbb{C}_\infty} 
\newcommand{\Fq}{\mathbb{F}_q} 
\renewcommand{\AA}{A} 

\newcommand{\ppi}{\widetilde{\pi}}

\newcommand{\pp}{p} 
\DeclareMathOperator{\Spec}{Spec}

\newcommand{\LA}{\Lambda}
\newcommand{\la}{\lambda}

\newcommand{\ppp}{\mathfrak{p}}
\newcommand{\GL}{\text{GL}_2(\AA)}
\newcommand{\HHH}{\Omega}

\newcommand{\MForms}{M_{k, m} (\GL)}
\newcommand{\SForms}{S_{k, m} (\GL)}

\DeclareMathOperator{\kere}{ker}

\DeclareMathOperator{\TT}{T}

\newcommand{\imp}{\Longrightarrow}

\newcommand{\delt}{\varDelta}

\DeclareMathOperator{\valu}{val}

\newcommand{\Ga}{\Gamma}

\newcommand{\dell}{\varDelta}

\DeclareMathOperator{\mini}{min}
\DeclareMathOperator{\anvarphi}{\widetilde{\varphi}}

\newcommand{\aalpha}{\xi}

\DeclareMathOperator{\degr}{deg}

\journal{arXiv.org}

\begin{document}

\begin{frontmatter}



\title{$\AA$-expansions of Drinfeld modular forms}


\author{Aleksandar Petrov}

\address{Texas A\&M University at Qatar}

\begin{abstract}
We introduce the notion of Drinfeld modular forms with $\AA$-expansions, where instead of the usual Fourier expansion in $t^n$ ($t$ being the uniformizer at `infinity'), parametrized by $n \in \mathbb{N}$, we look at expansions in $t_a$, parametrized by  $a \in A = \Fq [T]$. We construct an infinite family of eigenforms with $\AA$-expansions. Drinfeld modular forms with $\AA$-expansions have many desirable properties that allow us to explicitly compute the Hecke action.  
The applications of our results include: (i)~various congruences between Drinfeld eigenforms; 
(ii)~the computation of the eigensystems of Drinfeld modular forms with $\AA$-expansions;  
(iii)~examples of failure of multiplicity one result, as well as a restrictive multiplicity one result for Drinfeld modular forms with $\AA$-expansions; 
(iv)~examples of eigenforms that can be represented as `non-trivial' products of eigenforms; (v)~an extension of a result of B\"{o}ckle and Pink concerning the Hecke properties of the space of cuspidal modulo double-cuspidal forms for $\Gamma_1(T)$ to the groups $\text{GL}_2 (\Fq [T])$ and $\Gamma_0(T)$.
\end{abstract}


\end{frontmatter}

\section{Introduction}

Drinfeld modular forms are certain analogues of classical modular forms that were first introduced by Goss in \cite{Go1}, \cite{Go2}. Drinfeld modular forms have many properties that are similar to classical modular forms. However, there are several important differences that to this day make the theory of Drinfeld modular forms less understood in comparison to the classical theory. Arguably the most significant differences are: the apparent disconnect between the coefficients in the expansions at `infinity' on one hand, and the Hecke operators and eigenvalues on the other (the first indexed by the natural numbers, while the latter is indexed by the monic univariate polynomials over a finite field); the lack of transparent arithmetic significance of the coefficients due to that disconnect; the lack of diagonalizability of the Hecke action; and the lack of multiplicity one property (if we assume that the notion of multiplicity one remains the same as in the classical case). 

In the present work, we aim to address the points above by considering the concept of Drinfeld modular forms with $\AA$-expansions. Such forms have many desirable properties with respect to the Hecke algebra because the Hecke action on an $\AA$-expansion is easily computable. 

Before we introduce Drinfeld modular forms with $\AA$-expansions, we need to recall some notation from \cite{Gek}. Let $p$ be a rational prime and $q = p^e$. Let $\AA = \Fq[T]$ be the ring of univariate polynomials over $\Fq$, $\AA_+$ be the set of monic elements of~$\AA$, $K$ the fraction field of $\AA$, $K_\infty$ the completion of $K$ with respect to the valuation coming from $1/T$, and let $\CC$ be the completion of the algebraic closure of $K_\infty$. Let $\rho$ stand for the Carlitz module. We denote by $e_{\ppi \AA}$ the Carlitz exponential and by $\ppi$ a fixed choice for a period of the Carlitz module. If $\LA$ is a lattice, then its $n^\text{th}$ Goss polynomial will be denoted by $G_{n, \LA}(X)$. A primed sum $\sideset{}{'} \sum$ will denote that $0$ is omitted from the set over which we are summing.

Given $k \in \NN$, an integer $m$, $0 \leq m \leq q-1$, and a congruence subgroup $\Gamma \subset \GL$, we let $M_{k, m}(\Gamma)$ denote the space of Drinfeld modular forms for $\Gamma$ of weight $k$ and type $m$. The subspaces of $M_{k, m}(\Gamma)$ that consist of cuspidal and double-cuspidal Drinfeld modular forms will be denoted by $S_{k, m}(\Gamma)$ and $S_{k, m}^2(\Gamma)$, respectively. Every $f \in M_{k, m}(\GL)$ has a power series expansion in $t := t(z) = 1/e_{\ppi \AA}(\ppi z)$:
\[
    f(z) = \sum_{n = 0}^\infty a_n t^n, \qquad \qquad a_n \in \CC. 
\]
The expression on the right-hand side only converges for $z$ in a neighborhood of `infinity', but it determines $f$ uniquely for any $z$. Given $a \in \AA_+$ of degree~$d$, we let $t_a := t(az)$. One can show that $t_a$ can be expanded in a power series in $t$:
\[
	t_a  = \frac{t^{q^d}}{\rho_a (t^{-1}) t^{q^d}} = \frac{t^{q^d}}{1 + \cdots} = t^{q^d} + \cdots.
\]
We define $G_n (X) := G_{n, \ppi \AA}(X)$ which conforms with the notation in \cite[(3.4)]{Gek}. Property (iv) in \cite[(3.4)]{Gek} shows that $X \mid G_n (X)$ for any positive $n$. 

Consider the formal series indexed by $A_+$
\[
	\sum_{a \in A_+} c_a G_n (t_a) \in \CC [[t]],
\]
where $c_a \in \CC$. If $|c_a|$ has polynomial growth in $|a|$ for all but finitely many $a \in A_+$, then the series converges to a well-defined function on $\{z \in \HHH: |z|_i > 1 \}$. Indeed, for such $z$, Lemma (5.5) from \cite{Gek} shows that $|t| = |t(z)| \leq q^{-|z|_i}$. But then $X \mid G_n (X)$ implies that
\[
	|G_n (t_a)| \leq |t_a| = |t|^{|a|} \leq q^{-|a||z|_i} < q^{-|a|}.
\]
Therefore if $|c_a|$ has polynomial growth in $|a|$, then for $z$ with $|z|_i > 1$
\[
	\lim_{|a| \to \infty} |c_a G_n (t_a)| \to 0,
\]
and
\[
	\sum_{a \in A_+} c_a G_n (t_a) = \sum_{a \in A_+} c_a G_n \left (t(az) \right)
\]
converges. Such series are the main topic of the present paper:

 \begin{definition} \label{def: aexp}
A modular form $f \in \MForms$ is said to have an \emph{$\AA$-expansion} if there exists a positive integer $n$ and coefficients $c_0(f), c_a(f) \in \CC$  such that
\[
 f = c_0(f) + \sum_{a \in A_+} c_a(f) G_n(t_a).
\]
 Here the equality above is meant as an equality in $\CC[[t]]$ between the $t$-expansion of $f$ on the left side and the expression on the right side. We will call the integer~$n$ an \emph{$\AA$-exponent} of $f$ and the number $c_a = c_a(f)$ the \emph{$a^\text{th}$ coefficient} of $f$. Since $t \mid G_n (t_a)$ for any $n \in \NN$, $a \in \AA_+$, if $f \in \SForms$, then $c_0(f) = 0$. Theorem~\ref{uniq} below will show that if we fix the $\AA$-exponent~$n$, then the $\AA$-expansion is unique.\end{definition}

\begin{rmk} If we consider the action of scalar matrices on $f$ it is easy to see that $n \equiv m \bmod (q-1)$. We will show that if $f$ is a simultaneous eigenform, then the $\AA$-exponent $n$ of $f$ is unique. This will follow from our multiplicity one result for forms with $\AA$-expansions (Theorem \ref{aexpmultone}). For a~general Drinfeld modular form $f$ with $\AA$-expansion we do not know, but we strongly suspect, that the $\AA$-exponent $n$ is unique.
\end{rmk}

Examples of such forms have been known to Goss in the form of  Eisenstein series, which until recently appeared to be the only examples. In a recent paper~\cite{Lop1}, L\'{o}pez showed two additional examples of Drinfeld modular forms with $\AA$-expansions. We prove that there are infinitely many examples of cuspidal Drinfeld eigenforms with $\AA$-expansions. 

\begin{theorem} \label{mainthm}
 Let $k, n$ be two positive integers such that $k - 2n$ is a positive multiple of $(q-1)$ and $n \leq p^{\valu_p(k-n)}$. Then
\[
  f_{k, n} := \sum_{a \in \AA_+} a^{k-n} G_n(t_a) 
\] is an element of $S_{k, m}(\GL)$, where $m \equiv n \bmod (q-1)$.
\end{theorem}

We will prove this result in Section~\ref{the proof}.

Properties of $\AA$-expansions (Theorem~\ref{aexpeigenvalues} below) show that the form $f_{k, n}$ is an eigenform with eigensystem $\{ \la_\ppp = \wp^n \}$. Therefore, Theorem~\ref{mainthm} produces infinitely many examples of cuspidal Hecke eigenforms with explicit eigensystems. Theorem~\ref{mainthm} has other important consequences:
\begin{itemize}
\item it shows that the space of single-cuspidal modulo double-cuspidal Drinfeld modular forms is parametrized by forms with $\AA$-expansions for $\Gamma = \GL, \Gamma_1(T), \Gamma_0(T)$ (Theorem~\ref{cuspmoddcuspGL}, Example~\ref{ex18}, Theorem~\ref{missingtheorem}, respectively);
\item it provides examples of `non-trivial' congruences between eigenforms (Theorem~\ref{aexpcongruence});
\item it suggests interesting eigenproduct identities (Remark~\ref{cj1} and Theorem~\ref{products}) in the case of Drinfeld modular forms. 
\end{itemize}

\section{Properties of Drinfeld Modular Forms with $\AA$-expansions}

Before we turn to the proof of Theorem~\ref{mainthm}, we discuss the properties of Drinfeld modular forms with $\AA$-expansions that make them so useful.

Drinfeld modular forms with $\AA$-expansions were among the first concrete examples of Drinfeld modular forms considered by Goss \cite[Sec.~2]{Go2}:

\begin{example} [Eisenstein series]
Let $k$ be a positive integer and consider the (non-normalized) Eisenstein series:
\[ 
  E_k (z) :=    \frac{1}{\ppi^k} \sideset{}{'} \sum_{(a, b) \in \AA^2} \frac{1}{(a z + b)^k} = \sideset{}{'} \sum_{b \in \AA} \frac{1}{(\ppi b)^k} - \sum_{a \in \AA_+} \sum_{b \in \AA} \frac{1}{(\ppi a z + \ppi b)^k}. 
\]
It is known that $E_k \in M_{k, 0} (\GL)$. If $k \not \equiv 0 \bmod (q-1)$, then $E_k = 0$. Assume that $ k  \equiv 0 \bmod (q-1)$. Then $E_k$ is not cuspidal and we can renormalize so that its first non-zero coefficient is $1$ to obtain: 
\[
  g_k := \frac{1}{\delta_k} E_k =   1 - \frac{1}{\delta_k} \sum_{a \in \AA_+} G_k(t_a), \qquad \qquad \text{where~} \delta_k = \sideset{}{'} \sum_{b \in \AA} \frac{1}{(\ppi b)^k}. \]
This is our first family of examples of Drinfeld modular forms with $\AA$-expansions. Following \cite{Gek}, we define $g := g_{q-1}$. 
\end{example}

Until the work of L\'{o}pez \cite{Lop1} in 2010, the family of normalized Eisenstein series $\{ g_k \}_{ k \equiv 0 \bmod (q-1)}$ gave the only examples of Drinfeld modular forms with $\AA$-expansions. L\'{o}pez considered $\AA$-expansions with $\AA$-exponents $n = 1$ and $n=q-1$, i.e., $\AA$-expansions in terms of $G_1(t_a) = t_a$ and $G_{q-1}(t_a) = t_a^{q-1}$. 
L\'{o}pez showed that there are two additional examples: the theoretically important (see \cite[(5.12) \& (5.13)]{Gek}) forms $h$ and~$\dell$. L\'{o}pez proved that
\[
  h = \sum_{a \in \AA_+} a^q t_a, \qquad \qquad \dell = \sum_{a \in \AA_+} a^{q(q-1)} t_a^{q-1}.
\]
The reader should note that we have normalized both $h$ and $\dell$ so that their respective expansions have $1$ as the first non-zero coefficient.

Theorem~\ref{mainthm} shows that $h$ and $\dell$ are just two examples in a whole family of infinitely many Drinfeld modular forms that possess $\AA$-expansions. All of the new examples are cuspidal or double-cuspidal eigenforms and we will be able to explicitly compute their eigensystems. The first result that we will need is:

\begin{theorem} [Uniqueness of an $\AA$-expansion] \label{uniq}
 \[
   c_0 + \sum_{a \in \AA_+} c_a G_n(t_a) = c_0' + \sum_{a \in \AA_+} c_a' G_n(t_a) \ \imp \  c_a = c_a' \quad \forall a \in \AA_+ \cup \{ 0 \}.
 \]
\end{theorem}



\begin{proof} The case $n = q-1$ has been proved by L\'{o}pez in \cite[Thm.~3.1]{Lop2} and the same proof works for general $n$.  
\end{proof}

Next, we turn to Hecke properties of Drinfeld modular forms with $\AA$-expansions. Let $\ppp$ be a  non-zero prime ideal of $\AA$, let $\wp$ be its unique monic generator and let
\[S_\ppp := \{ \beta \in \AA: \degr (\beta) < \degr (\wp) \}.\]
Following \cite[Sec.~3]{Go2} and \cite[Sec.~7]{Gek}, we define the $\ppp^\text{th}$ Hecke operator $\TT_\ppp$:
\[
   \TT_\ppp f (z) := \wp^k f(\wp z) + \sum_{\beta \in S_\ppp} f \left ( \frac{z + \beta}{\wp} \right), \qquad  \text{where $f \in M_{k, m} (\GL)$.}
\]
A Drinfeld modular form $f$ is called a \emph{simultaneous eigenform} or simply an \emph{eigenform}, if there exist  $\la_\ppp$'s in $\CC$ such that
\[
   \TT_\ppp f = \la_\ppp f, \qquad \qquad \forall \ppp \in \Spec (\AA) \setminus \{ 0 \}.
\]
For such an $f$ the values $\{ \la_\ppp \}_{ \ppp \in \Spec(\AA) \setminus \{0 \}}$ will be called the \emph{eigensystem}~of~$f$.
Goss computed the action of $\TT_\ppp$ on the $t$-expansion, which in our notation gives (see \cite[(7.3)]{Gek}):
\[
  \TT_\ppp \left ( \sum_{n = 0}^\infty a_n t^n \right) = \sum_{n = 0}^\infty a_n t_{\wp}^n + \sum_{n = 0}^\infty a_n G_{n, \ppp} ( \wp t).
\]
In the formula above, $G_{n, \ppp}(X)$ is the $n^\text{th}$ Goss polynomial for the lattice $\kere \rho_\ppp$. Drinfeld modular forms with $\AA$-expansions behave even better with respect to the action of $\TT_\ppp$ as the next result shows:

\begin{theorem} \label{aexpeigenvalues} Suppose that $f \in \SForms$ is an eigenform for $\TT_\ppp$ with eigenvalue $\la_\ppp$ and that $f$ has an $\AA$-expansion with exponent~$n$. Then $\la_\ppp = \wp^n$ and $c_\ppp(f) = \wp^{k-n} c_1(f)$. 
\end{theorem}

\begin{proof} Since $f$ and $\wp$ are fixed, we let $c_a = c_a(f)$.  We compute the Hecke action
\[
 \begin{aligned}
  \TT_\ppp f & = \wp^k \sum_{a \in \AA_+} c_a G_n (t_{\wp a}) + \sum_{\beta \in S_\ppp} \sum_{a \in \AA_+} c_a G_n \left (t_a \left (\frac{z + \beta}{\wp} \right ) \right ) \\
  & = \wp^k \sum_{a \in \AA_+} c_a G_n(t_{\wp a}) + \frac{1}{\ppi^n} \sum_{\beta \in S_\ppp} \sum_{a \in \AA_+} \sum_{b \in \AA} \frac{c_a \wp^n}{(a z + a \beta + b \wp)^n} \\
 & = \wp^k \sum_{a \in \AA_+} c_a G_n(t_{\wp a}) + \frac{1}{\ppi^n} \sum_{a \in \AA_+} \sum_{b \in \AA} c_a \wp^n \sum_{\beta \in S_\ppp} \frac{1}{(az + a \beta + b \wp)^n}. 
\end{aligned}
\]
If $(a, \wp) = 1$, then the map $\AA \times S_\ppp \to \AA$, which sends $(b, \beta)$ to $a \beta + b \wp$, is a bijection. The inner double sum is absolutely convergent, therefore by rearranging we obtain
\[
 \sum_{b \in \AA}  \sum_{\beta \in S_\ppp} \frac{1}{(az + a \beta + b \wp)^n} = \sum_{b \in \AA} \frac{1}{(az + b)^n} = G_n(t_a).
\]
If $(a, \wp) = \wp$, then the map $\AA \times S_\pp \to \AA$, which sends $(b, \beta)$ to $a \beta + b \wp$, is surjective, and every output has a number of preimages which is divisible by~$q$. Hence
\[
\sum_{b \in \AA} \sum_{\beta \in S_\ppp} \frac{1}{(az + a \beta + b \wp)^n} = 0.
\]
It follows that
\[
 \TT_\ppp f = \wp^k \sum_{a \in \AA_+} c_a G_n(t_{\wp a}) + \wp^n \sum_{a \in \AA_+, (a, \wp) = 1} c_a G_n(t_a). 
\]
Noting that $\TT_\ppp f = \la_\ppp f$ and comparing coefficients in the $\AA$-expansions, we see that if there exists $a \in \AA_+$ such that $(a, \wp) = 1$ and $c_a \neq 0$, then $\la_\ppp = \wp^n$. But if all the $c_a$ satisfying $(a, \wp) = 1$ are zero, then again looking at the $\AA$-expansions on both sides we see that $f$ cannot be an eigenform for $\TT_\ppp$. Indeed, by the computation above 
\[
           f = \sum_{a \in \AA_+} c_{\wp a} G_n(t_{\wp a}) \imp \la_\ppp \sum_{a \in \AA_+} c_{\wp a} G_n(t_{\wp a}) = \wp^k \sum_{a \in \AA_+} c_{\wp a} G_n(t_{\wp^2 a}),
\]
which contradicts the uniqueness of the $A$-expansion. By comparing $\wp^\text{th}$ coefficients on both sides, we get
\[
  c_\wp = \frac{\wp^k}{\la_\ppp} c_1 = \wp^{k-n} c_1.
\]
\end{proof}

\begin{coro} \label{acoro} Assume that $f \in \SForms$ is a modular form that possesses an $\AA$-expansion with exponent $n$. Let $a = \prod_{i = 1}^\nu \wp_i^{e_i}$ for distinct monic primes $\wp_i$. If $f$ is an eigenform for $T_{\ppp_1}, \ldots, T_{\ppp_\nu}$, then $c_a(f) = a^{k-n} c_1(f)$.
\end{coro}

\begin{proof} This follows by induction on the factorization of $a$. \end{proof}

Classically any Hecke eigenform for $\text{SL}_2(\ZZ)$ is determined up to a multiplicative constant by its eigensystem. This is known as the \emph{multiplicity one} property (usually one speaks of the multiplicity one property of cusp forms). The analogous multiplicity one property is not true for Drinfeld modular forms for $\GL$. For instance, $g$, $g^q \delt$, $\delt$ are all eigenforms with eigensystem $\{ \la_\ppp = \wp^{q-1} \}$ (see \cite[Cor.~2.2.4, 2.2.5]{Go1}). Indeed, Theorem~\ref{mainthm} provides infinitely many counterexamples:

\begin{example} Let $n = p^r$. If $u$ is a positive integer, then the pair $(p^r (2 + u(q-1)), p^r)$ satisfies the hypothesis of Theorem~\ref{mainthm}. And therefore, the family
\[
	f_{p^r (2 + u(q-1)), p^r} = \sum_{a \in \AA_+} a^{p^r (1 + u(q-1))} t_a^{p^r} \in S_{p^r (2 + u(q-1)), p^r} (\GL), \qquad u \in \NN
\]
consists of eigenforms with eigensystem $\{ \la_\ppp = \wp^{p^r} \}$. 

The family $\{ f_{p^r (2 + u(q-1)), p^r} \}_{u \in \NN}$ consists of $p^r$-powers of the forms from the family $\{ f_s \}_{s \in \ZZ_{\geq 0}}$ defined in Definition~\ref{specialfam} below. 

Next, let $n$ be a fixed positive integer, which is not a $p^\text{th}$-power. Put $\nu = \lceil \log_p (n) \rceil$. Let $u_0$ be a positive integer, which satisfies the congruence $n \equiv u_0 (1-q) \bmod p^{\nu}$. Then, for any integer $u \geq 0$, the pair $(2n + u_0 (q-1) + p^\nu u (q-1), n)$ satisfies the hypothesis of Theorem~\ref{mainthm}. Therefore, the family
\[
	f_{2n + u_0 (q-1) + p^\nu u (q-1), n} = \sum_{a \in \AA_+} a^{n + u_0 (q-1) + p^\nu u (q-1)} G_n (t_a) , \qquad u \geq 0\]
consists of eigenforms with eigensystem $\{ \la_\ppp = \wp^n \}$. The family of eigenforms $\{ f_{2n + u_0 (q-1) + p^\nu u (q-1), n}\}_{u \in \ZZ_{\geq 0}}$ can be obtained from the family $\{ f_s \}_{s \in \ZZ_{\geq 0}}$ (Definition~\ref{specialfam} below) by using divided derivative (see \cite[p.~5]{BoPe}).
\end{example}

 Since the classical version of multiplicity one fails for Drinfeld modular forms, Gekeler asked if a Drinfeld eigenform for $\GL$ is determined up to a multiplicative constant by its eigenvalues and its weight. We do not know if the answer to Gekeler's question is positive or negative (in general) when $f$ is an eigenform, or even a cuspidal eigenform, for~$\GL$. There are multiplicity one results due to Armana \cite[Thm.~7.7]{Cec} for forms of low weight for $\GL$. The situation is much more favorable if we assume that the eigenform has an $\AA$-expansion. The following result (which is an immediate consequence of Corollary~\ref{acoro}) shows that a cuspidal Drinfeld eigenform with an $\AA$-expansion is uniquely determined by its  eigensystem $\{ \la_\ppp \}$ and its weight $k$ (as predicted by a positive answer to  Gekeler's question):

\begin{theorem}[Multiplicity One for modular forms with $A$-expansions] \label{aexpmultone} \qquad \qquad 

 If $f \in \SForms$ is an eigenform that possesses an $\AA$-expansion with exponent~$n$, then 
\[
  f = \sum_{a \in \AA_+} a^{k-n} G_n(t_a).
\]
Therefore, $f$ is determined uniquely by its weight $k$ and the eigenvalues $\{ \la_\ppp = \wp^n \}$.
\end{theorem}

\begin{example}[Non-examples]
Theorem~\ref{aexpmultone} shows that eigenforms with $\AA$-expansions can only have very restrictive types of eigensystems. In particular, not every Drinfeld eigenform can have an $\AA$-expansion. For instance, the form $h^2 g \in S_{10, 0} (\GL)$ is an eigenform when $q = 3$ (since $S_{10, 0} (\GL)$ is one-dimensional), but computations for $\ppp$ of degree $\leq 4$ show that $\la_\ppp \neq \wp^n$ for any $n$.

The natural question is: \emph{Does every cuspidal eigenform with eigensystem $\{ \la_\ppp = \wp^n \}$ possess an $\AA$-expansion?}  We strongly suspect that the answer is \emph{No}. The example that we have in mind is $h^2 g^2 \in S_{12, 0}^2 (\GL)$ when $q = 3$ (this is an eigenform, since $S_{12, 0}^2 (\GL)$ is one-dimensional). The only reason that we cannot be completely certain is that we cannot show that the eigenform $h^2 g^2$ has eigenvalues $\la_\ppp = \wp^4$ for all $\ppp$. We have verified that for $\ppp$ of degree $\leq 4$, $\TT_\ppp h^2 g^2 = \wp^4 h^2 g^2$, and $h^2 g^2$ does not have an $\AA$-expansion.
\end{example}

\section{Consequences of Theorem \ref{mainthm}}

\subsection{Single-Cuspidal Forms that are not Double-Cuspidal} Theorem~\ref{mainthm} allows us to define a special family of Drinfeld modular forms that turns out to parametrize the space of strictly single-cuspidal Drinfeld modular forms for $\GL$.

\begin{definition} \label{specialfam} If $s \in \ZZ_{\geq 0}$, then we define
\[
  f_s := f_{q + 1 + s(q-1), 1} = \sum_{a \in \AA_+} a^{q + s(q-1)} t_a.
\]
\end{definition}

It follows from Theorem~\ref{mainthm} that $f_s$  is an element of the space of cuspidal forms $S_{q + 1+s(q-1), 1} (\GL)$, which is not in $S_{q +1 +s(q-1), 1}^2 (\GL)$.

 Since the weights $k = q + 1 + s(q-1), s \geq 0$, are precisely the weights for which $S_{k, 1} (\GL) / S_{k, 1}^2 (\GL) \neq 0$, this shows that:
 
 \begin{theorem} \label{cuspmoddcuspGL} The space 
\[
    S_{k, 1} (\GL) / S_{k, 1}^2(\GL)
\]
is diagonalizable, and the eigenforms for $S_{k, 1} (\GL)$ whose images form a coset eigenbasis have eigenvalues $\la_\ppp = \wp$. 
\end{theorem}

 In \cite[Ex.~15.7]{Boe1}, B\"{o}ckle shows\footnote{The reader should be aware that B\"{o}ckle uses a different normalization for $\TT_\ppp$ and with his normalization the eigenvalues are all equal to $1$, which corresponds to $\la_\ppp = \wp$ in our notation.} that the same result holds for $\Gamma_1(T)$. We will see below that we can reprove B\"{o}ckle's result by using $\AA$-expansions and extend it to $\Gamma_0(T)$ as hinted by \cite[Rem.~12.9 \& Ex.~15.7]{Boe1}. 

To that end, let $\Ga$ be either $\Gamma_1(T)$ or $\Gamma_0(T)$. We have two natural maps from $M_{k, m}(\GL)$ to $M_{k, m}(\Ga)$, which respect cuspidality and double-cuspidality:
\[ 
  \begin{aligned}
  &  \iota: M_{k, m}(\GL) \to M_{k, m}(\Ga) : f(z) \mapsto f(z), \\
  & \iota_T : M_{k, m} (\GL) \to M_{k, m}(\Ga): f(z) \mapsto F(z) = f(Tz).
  \end{aligned}
\]
The effect of $\iota_T$ on $\AA$-expansions is as follows:
\[
   \iota_T \left (c_0 + \sum_{a \in \AA_+} c_a G_n (t_a) \right) = c_0 + \sum_{a \in \AA_+} c_a G_n (t_{a T}).
\]
The proof of Theorem~\ref{aexpeigenvalues} shows that $\iota_T f_{k, n}$ remains an eigenform away from the level. That is, if $\ppp \neq (T)$, then
\[
	\TT_\ppp \iota_T f_{k, n} = \ppp^n \iota_T f_{k, n}.
\]


\begin{example} \label{ex18}
Example 15.7 in\footnote{These are examples due to B\"{o}ckle and Pink.}  \cite{Boe1} shows that the two-dimensional quotient space $S_{k, 0} (\Gamma_1(T)) / S_{k, 0}^2 (\Gamma_1(T))$, is always diagonalizable with respect to the Hecke algebra away from $T$. And any eigenform in this space has eigenvalues\footnote{With B\"{o}ckle's normalization the eigenvalues are actually $\la_\ppp = 1$.} $\la_\ppp = \wp$ for $\ppp \neq T$. 

The use of $\AA$-expansions allows us to see that this also follows without using the cohomological tools developed in \cite{BP}. Indeed, if $k \equiv 1 \bmod (q-1)$, then write $k = s (q-1) + 1$.  Using $\iota$ and $\iota_T$ to induce $f_s$ to $\Gamma_1(T)$, we have two linearly independent forms 
\[
\iota(f_s), \iota_T(f_s) \in S_{k, 0} (\Gamma_1(T)) / S_{k, 0}^2 (\Gamma_1(T)),
\] with the same eigensystem away from the level $\{ \la_\ppp = \wp \}_{\ppp \neq (T)}$. 
\end{example}

The argument applies verbatim to $\Gamma_0(T)$, and we obtain:
\begin{theorem} \label{missingtheorem}
 The quotient space $S_{k, 1} (\Gamma_0(T)) / S_{k, 1}^2 (\Gamma_0(T))$ has a basis of eigenforms away from $T$. Each element of this basis has eigensystem $\{ \la_\ppp = \wp \}_{\ppp \neq (T)}$. \qed
\end{theorem}

B\"{o}ckle observed in Remark~12.9 and Example~15.7 of \cite{Boe1} that the quotient space may be generated by Poincar\'{e} series. Our results show that this space is generated by forms with $\AA$-expansions for $\Ga = \GL, \Gamma_0(T)$ and $ \Gamma_1(T)$, i.e., that forms with $\AA$-expansions parametrize the quotient spaces for these congruence subgroups. We do not know if this happens for general congruence subgroups.

\subsection{The Family $F_\nu$} 

\begin{definition}
Given $\nu \in \NN$ we define
\[
  F_\nu := \sum_{a \in \AA_+} a^{q^\nu} t_a.
\]
\end{definition}

Since $(q-1) \mid (q^\nu - 1)$ it follows from Theorem~\ref{mainthm} that  \[ \{ F_\nu \in S_{q^\nu + 1, 1}(\GL) \}_{\nu \in \NN} \subset \{f_s \}_{s \in \ZZ_{\geq 0}}. \] The family $\{F_\nu\}_{\nu \in \NN}$ satisfies a recursive formula, which is similar to the formula for the subfamily $\{g_{q^k - 1}\}_{k \in \NN}$ of Eisenstein series given in \cite[Prop.~6.9]{Gek}.

\begin{theorem}
 We have $F_1 = h$, $F_2 = h g^q$ and the recursive formula for $\nu \geq 2$
\[
 F_\nu = \frac{g^q}{h^{q-1}} F_{\nu-1}^q - \frac{[\nu -2]^{q^2}}{h^{q-1}} F_{\nu -2}^{q^2},
\]
where $[i] := T^{q^i} - T$.
\end{theorem}

\begin{proof}
 Following Pellarin\footnote{Pellarin considers $\mathbb{E}$ in \cite[Sec.~3]{Pel1}, but the formula that we have used to define $\mathbb{E}$ is shown in \cite[Cor.~5]{Pel2}.}, we define
\[
 \mathbb{E}(z, u) = \sum_{a \in \AA_+} a(u) t_a \in \CC[[t, u]],
\]
where $u$ is a new variable independent of $t$ and $T$.
 Let $\anvarphi$ be the map that fixes $u$ and acts on the elements of $\CC[[t]]$ by $x \to x^q$ (the partial Frobenius). The space $\CC [[t, u]]$ also has the usual Frobenius, $\varphi$, which acts as $x \to x^q$ on every element of $\CC[[t, u]]$. By definition of $\varphi$ and $\anvarphi$, we have
\[
  (\varphi \circ \anvarphi^{-1})^\nu \mathbb{E}(z, u)_{\mid_{u = T}} = F_\nu.
\]
Pellarin has shown (see \cite[Prop.~9]{Pel1}) that $\mathbb{E}$ satisfies the $\anvarphi$-difference equation 
\[
 \anvarphi^2 \mathbb{E} = \frac{1}{u - T^{q^2}} (-h^{q-1} \mathbb{E} + g^q \anvarphi \mathbb{E}),
\]
which we rewrite as
\[
 \mathbb{E} = \frac{g^q}{h^{q-1}} \anvarphi \mathbb{E} - \frac{(u-T^{q^2})}{h^{q-1}} \anvarphi^2 \mathbb{E}.
\]
Applying $(\varphi \circ \anvarphi^{-1})^\nu$ to both sides and plugging in $u = T$, we get the recursion
\[
  F_\nu = \frac{g^q}{h^{q-1}} F_{\nu-1}^q - \frac{T^{q^\nu} - T^{q^2}}{h^{q-1}} F_{\nu-2}^{q^2}.
\]
\end{proof}

\begin{rmk} Our computations suggest that the following equality holds for~$j \leq q$:
\[
  \mathbb{E}^j = \sum_{a \in \AA_+} a(u)^{j} t_a^j.
\]
If we assume this conjectural equality between $\mathbb{E}^j$ and the expression on the right, then the $\anvarphi$-difference equation for $\mathbb{E}^j$ will allow us to prove the recursive relations (which we have also observed computationally) among the Drinfeld modular forms 
\[
\Phi_{\nu, j} := \sum_{a \in \AA_+} a^{j q^\nu} t_a^j
\] for different $\nu's$, where $j \leq q$. We hope to return to this in future work.
\end{rmk}

\begin{example} Using the recursion, one easily computes:
\[
 \begin{aligned}
  F_3 = \sum_{a \in \AA_+} a^{q^3} t_a & =  h g^{q^2 + q}   - [1]^{q^2} h^{q (q-1) + 1}, \\
  F_4 = \sum_{a \in \AA_+} a^{q^4} t_a & = h g^{q^3 + q^2 + q} - [2]^{q^2} h^{q (q-1)+ 1} g^{q^3} - [1]^{q^3} h^{q^2 (q-1) +1} g^q, \\
  F_5 = \sum_{a \in \AA_+} a^{q^5} t_a & =   h g^{q^4 + q^3 + q^2 + q}   - [3]^{q^2} h^{q(q-1) + 1} g^{q^4 + q^3}\\ 
& \  \ \ - [2]^{q^3} h^{q^2(q-1) + 1} g^{q^4 + q}  - [1]^{q^4} h^{q^3(q-1) + 1} g^{q^2 + q} \\
 & \ \ \ \ + [1]^{q^4}  [3]^{q^2} h^{(q^3 + q)(q-1) + 1}. \\
 \end{aligned}
\]
\end{example}

\subsection{Congruences Between Eigenforms} \label{congruences}
Another classically important topic is that of congruences between modular forms. Several results have appeared that seem to mirror the classical situation (see \cite[Sec.~12]{Gek} and \cite{Vincent}). It turns out that we can use Theorem~\ref{mainthm} to obtain a new result regarding congruences between Drinfeld eigenforms, because the $\AA$-expansions make it possible to easily observe congruences. 

\begin{definition}
Let $k, n$ be two positive integers that satisfy the hypothesis of Theorem~$\ref{mainthm}$. For any integer $l \geq 0$, define
\[
	F_{k, n, l} := f_{(k-n)q^l + n, n} =  \sum_{a \in \AA_+} a^{(k-n) q^{l}} G_n (t_a) \in S_{(k-n) q^l + n, n}(\GL).
\]
\end{definition}

\begin{theorem} \label{aexpcongruence}
Let $\nu_0 = \valu_p(k-n)$, and let $\nu$ be any non-negative integer.  

If $\ppp$ is any prime of degree $d$ with $d > \log_q(n)$, then
\[
 F_{k, n, d + \nu} \equiv F_{k, n, \nu} \bmod \ppp^{q^\nu p^{\nu_0}}.
\]
\end{theorem}
 
Note that the weights of the forms in the congruence are $(k-n)q^{d+\nu} + n$ and $(k-n)q^\nu + n$, respectively.

\begin{proof} 

Let $\ppp$ be a prime of degree $d$. Then $\ppp \mid (a^{q^d} - a)$ for all $a \in \AA$. Therefore, we have
\[ 
\ppp^{q^\nu p^{\nu_0}}  \mid \left ( a^{(k-n)q^{d + \nu}} - a^{(k-n)q^\nu} \right )
\]
 for all $a \in \AA$. Because of the $\AA$-expansions on both sides, the congruence 
\[
 F_{k, n, d + \nu} \equiv F_{k, n, \nu} \bmod \ppp^{q^\nu p^{\nu_0}}
\]
will follow if we can prove that $\ppp$ does not divide the denominators of the coefficients of $G_n (X)$ (note that if $a \in \AA_+$, then $t_a$ has no denominators in its $t$-expansion). Since we are taking $d > \log_q(n)$, in each case we are considering, this follows from formula (3.8) in \cite{Gek}, combined with the fact (see \cite[Sec.~2.5]{Tha}) that the $q^l$-th coefficient of $e_{\ppi \AA} (z)$ equals the reciprocal of the product of all monic polynomials of degree $l$. \end{proof}

\begin{rmk} \label{bettercongruence}
 Note that we actually have
\[
 [d]^{q^\nu p^{\nu_0}}  \mid \left ( a^{(k-n)q^{d + \nu}} - a^{(k-n)q^\nu} \right ), \qquad \qquad \forall a \in A,
\]
where $[d] = T^{q^d} - T$ is the product of all monic primes of degree dividing $d$. Therefore, if $[d]$ is relatively prime to the denominators of the coefficients of $G_{n} (X)$, then we obtain the stronger congruence
\[
  F_{k, n, d + \nu} \equiv F_{k, n, \nu} \bmod [d]^{q^\nu p^{\nu_0}}.
\]
\end{rmk}

\begin{rmk}
 The proof of Theorem~\ref{aexpcongruence} is deceptively simple however this is because the $\AA$-expansions have packaged the $t$-expansions on both sides in a special way.  It is unclear how to prove the result of the previous theorem without observing the $\AA$-expansions, i.e., by just looking at the $t$-expansions or at the expressions in terms of $h$ and $g$.
\end{rmk}

\begin{rmk} One should note that Theorem \ref{aexpcongruence} gives congruences in two directions: for varying $d$ and fixed $\nu$, and for fixed $d$ and varying $\nu$. We will give examples of both.
\end{rmk}

 Some of the results before the present work, particularly $g_{q^d-1} \equiv 1 \bmod [d]$ from \cite[Prop.~6.12]{Gek}, were also proven by using the  $\AA$-expansions of Eisenstein series. It is interesting to see if there are other congruences that come from $\AA$-expansions of forms that are not eigenforms.

We end this subsection with several examples of congruences obtained from Theorem~\ref{aexpcongruence}.

\begin{example} First, we present examples with increasing $d$ and fixed $\nu = 0$. We have
\[
F_{q + 1, 1, d} = F_{d+1} = \sum_{a \in \AA_+} a^{q \cdot q^d} t_a ,\] so that $h = F_{q +1, 1, 0}$, $h g^q = F_{q + 1, 1, 1}$, etc.

This gives
\[
\begin{aligned}
  & h \equiv h g^q = F_{2} \bmod [1]^q, \\
  & h \equiv h g^{q^2 + q} - [1]^{q^2} h^{q(q-1) + 1} = F_{3} \bmod [2]^q, \\
  & h  \equiv h g^{q^3 + q^2 + q} - [2]^{q^2} h^{q(q-1) + 1} g^{q^3} - [1]^{q^3} h^{q^2(q-1) + 1} g^q = F_4 \bmod [3]^q . \\
\end{aligned}
\]

Another family for which we obtain congruences is 
\[
 F_{q(q-1) + 1, 1, d} = \sum_{a \in \AA_+} a^{q(q-1)q^d} t_a^{q-1}, \qquad \qquad d \geq 0,
\]
where $\delt = F_{q(q-1)+1, 1, 0}$. We have the congruences
\[
\begin{aligned}
  & \delt \equiv \delt g^{q^2 - q} = F_{q(q-1)+1, 1, 1} \bmod [1]^q , \\
  & \delt \equiv \delt g^{q^3 - q} + [1]^{q^2} \delt^{q+1} g^{q^3 - q^2 - 2 q} + [1]^{(q-1)q^2} \delt^{q^2 - q +1} = F_{q(q-1)+1, 1, 2} \bmod [2]^q. \\
\end{aligned}
\]
Since $G_1(X) = X$ and $G_{q-1} (X) = X^{q-1}$, we are in the situation described in Remark~\ref{bettercongruence}. We note that we cannot improve the congruence to $\bmod [d+1]$ (i.e., to $\mod \ppp$ with $\ppp$ of degree $d+1$) because
\[
  h \not \equiv h g^q \bmod [2], \qquad h \not \equiv h g^{q^2 + q} - [1]^{q^2} h^{q(q-1) + 1} \bmod [3].
\]
\end{example}

\begin{example}
 Let us fix $d = 1$ and let $\nu$ vary.  

 Notice that $F_{q+1, 1, 1 + \nu} = F_{2 + \nu}$. Then we have

\[
 \begin{aligned}
  F_5 & =   ( h g^{q^4 + q^3 + q^2 + q}   - [3]^{q^2} h^{q(q-1) + 1} g^{q^4 + q^3} - [2]^{q^3} h^{q^2(q-1) + 1} g^{q^4 + q} \\ 
      & ~ \ \    - [1]^{q^4} h^{q^3(q-1) + 1} g^{q^2 + q}   + [1]^{q^4}  [3]^{q^2} h^{(q^3 + q)(q-1) + 1}) \\
      & \equiv (h g^{q^3 + q^2 + q} - [2]^{q^2} h^{q (q-1)+ 1} g^{q^3} - [1]^{q^3} h^{q^2 (q-1) +1} g^q)   \bmod [1]^{q^2 \cdot q} \\
      & = F_4
 \end{aligned}
\]

We can also see $F_6 \equiv F_5 \bmod [1]^{q^3 \cdot q}$, $F_7 \equiv F_6 \bmod [1]^{q^4 \cdot q}$, $\ldots$. 
\end{example}

\subsection{Eigenproducts} In \cite{Gek1}, Gekeler proved\footnote{Actually Gekeler derived a product expansion for $\delt$. The result for $h$ follows immediately from that.} that $h$ has a product expansion that is indexed by the monic polynomials
\[
  h = t \prod_{a \in \AA_+} \psi_a (t)^{q^2 - 1},
\]
where $\psi_a$ is the $a^\text{th}$ inverse cyclotomic polynomial $\psi_a (X) := \rho_a (X^{-1}) X^{q^d}$ (see \cite[Eq.~4.6]{Gek}).

Theorem~\ref{mainthm} allows us to show that there are identities between $\AA$-expansions and product expansions indexed by $\AA_+$:

\begin{theorem}  \label{products} If $1 \leq j \leq q$, then
\[
 h^j = \sum_{a \in \AA_+} a^{q j} t_a^j = t^j \prod_{a \in \AA_+}  \psi_a(t)^{(q^2-1)j} .
\]
In particular, we have
\[
  h = \sum_{a \in \AA_+} a^q t_a , \qquad \qquad  \delt = h^{q-1} = \sum_{a \in \AA_+} a^{q(q-1)} t_a^{q-1}.
\]
\end{theorem}

\begin{proof}
We know that $h^j$ as well as the claimed $\AA$-expansion are in the one-dimensional space $S_{j(q+1),j} (\GL)$ by Theorem~\ref{mainthm}. Comparing the first non-zero coefficient of the $t$-expansions on both sides, the claimed equality follows.
\end{proof}

\begin{rmk} We remark that while the relations
\[
  h^j = t^j \prod_{a \in \AA_+} \psi_a(t)^{(q^2 - 1)j}
\]
are immediate from the product formula for $h$, the equations that follow from Theorem~\ref{products}
\[
  \left ( \sum_{a \in \AA_+} a^q t_a \right )^j = \sum_{a \in \AA_+} a^{qj} t_a^j, \qquad \qquad 1 \leq j \leq q
\]
are non-trivial and imply relations between the coefficients of the $t$-expansions on both sides.
\end{rmk}



 \begin{rmk} \label{cj1} Computer experimentations suggest that Theorem~\ref{products} is part of a more general phenomenon. Namely, if $G_n (X) \cdot G_{n'} (X) = G_{n + n'}  (X)$, then there exist weights $k, k'$ such that the pairs $(k, n), (k', n'), (k+k', n+n')$ satisfy the hypothesis of Theorem~\ref{mainthm}, and, for all $l, l' \in \ZZ_{\geq 0}$, the product
 \[
 	\left ( \sum_{a \in A_+} a^{q^l (k-n)} G_n (t_a) \right) \cdot \left ( \sum_{a \in A_+} a^{q^{l'} (k' -n')} G_{n'} (t_a) \right)
\]
equals
\[
	\sum_{a \in A_+} a^{q^l (k-n) + q^{l'} (k' - n')} G_{n + n'} (t_a).
\]
Given $n, n'$ and $q$ such that $G_n (X) \cdot G_{n'} (X) = G_{n+n'} (X)$, there could be more than one pair of integers $(k, k')$ that works, as Example~\ref{20} below shows. Because a Drinfeld modular form of type $k$ and weight $m$ is uniquely determined by the first $i$ coefficients in its $t$-expansion, with $i \leq \frac{k}{q+1} + 1$, we can verify the equality above case by case. We present several examples for various $q$. 	
\end{rmk}



\begin{example} \label{20} Let $q = 3$. Then our considerations suggest that
\[
	\begin{aligned}
	& \left ( \sum_{a \in A_+} a^{3 \cdot 3^l} t_a \right) \cdot \left ( \sum_{a \in A_+} a^{6 \cdot 3^{l'}} t_a^2 \right) & = \sum_{a \in A_+} a^{3 \cdot 3^l + 6 \cdot 3^{l'}} t_a^3, \\
	& \left ( \sum_{a \in A_+} a^{3 \cdot 3^l} t_a \right) \cdot \left (\sum_{a \in A_+} a^{12 \cdot 3^{l'}} t_a^2 \right) & = \sum_{a \in A_+} a^{3 \cdot 3^l + 12 \cdot 3^{l'}} t_a^3, 
	\end{aligned}
\]
for all $l, l' \in \ZZ_{\geq 0}$.

We have verified the equalities for $l, l' \leq 4$.
\end{example}

\begin{example}
Let $q = 3$. Then $G_7(X) \cdot G_8(X) = G_{15}(X)$. By the procedure in Remark~\ref{cj1}, we can prove that
\[
  \left ( \sum_{a \in \AA_+} a^9 G_7(t_a) \right ) \cdot \left ( \sum_{a \in \AA_+} a^{18} G_8(t_a) \right) = \sum_{a \in \AA_+} a^{27} G_{15} (t_a).
\]

Let $q = 4$. Then $G_7(X) \cdot G_4(X) = G_{11}(X)$ and we have
\[
 \left ( \sum_{a \in \AA_+} a^{16} G_7(t_a) \right ) \cdot \left ( \sum_{a \in \AA_+} a^{16} G_4 (t_a) \right ) = \sum_{a \in \AA_+} a^{32} G_{11}(t_a).
\]
\end{example}

\begin{rmk} The condition $G_n(X) \cdot G_{n'}(X) = G_{n + n'}(X)$ is necessary even if the pairs $(k, n), (k', n'), (k + k', n + n')$ all satisfy the hypothesis of Theorem~\ref{mainthm}, as the example
 \[
   \left ( \sum_{a \in \AA_+} a^{18} G_4(t_a) \right ) \cdot \left ( \sum_{a \in \AA_+} a^{18} G_6(t_a) \right ) \neq \sum_{a \in \AA_+} a^{36} G_{10} (t_a),
 \]
when $q = 3$,  shows.
\end{rmk}

\begin{example} \label{nn1} Remark~\ref{cj1} does not account for all examples of equalities between an $A$-expansion and a product of $A$-expansions that we have found. We have verified (when $q = 3$) that
\[
	\left ( \sum_{a \in A_+} a^5 t_a \right) \cdot \left ( \sum_{a \in A_+} a^7 t_a \right) = \sum_{a \in A_+} a^{12} t_a, 
\]
however
\[
	\left ( \sum_{a \in A_+} a^{5 \cdot 3} t_a \right) \cdot \left ( \sum_{a \in A_+} a^7 t_a \right) \neq \sum_{a \in A_+} a^{5 \cdot 3 + 7} t_a.
\]
We have not found other exceptions to Remark~\ref{cj1} in our computations. We suspect that such exceptions are forced by dimensional reasons. For instance, the form
\[
	\sum_{a \in A_+} a^{12} t_a^2
\]
generates the one-dimensional space of double-cuspidal forms of weight $14$.
\end{example}

\begin{rmk} \label{nontrivialeigenproducts}
 Remark~\ref{cj1} and Theorem \ref{aexpmultone} give examples of eigenforms that can be represented as products of eigenforms. Classically this rarely happens and such products have been explicitly determined (see \cite{Jo1}). In contrast to the classical case, in the case of Drinfeld modular forms we can have high order vanishing at the cusps. In the case of Drinfeld modular forms, one `trivial' way of obtaining infinitely many such products is to take $p^\text{th}$ powers of known eigenforms (for example, $h, h^p, h^{p^2}, \ldots$).  Our results yield `non-trivial' examples of such eigenproducts. 
It is interesting to see if Remark \ref{cj1}, together with some exceptional cases like Example~\ref{nn1}, and `trivial' products are the only eigenproducts in the Drinfeld setting.
\end{rmk}

\section{The Proof of Theorem \ref{mainthm}} \label{the proof}

Throughout this section, we will assume that $k$ and $n$ are positive integers such that $k \geq 2 n$, $k - 2n \equiv 0 \bmod (q-1)$ and $n \leq p^{\valu_p(k-n)}$. We use the standard notation $\AA_{<d} := \{ a \in \AA : \degr(a) < d \}$, $\AA_{<d+} := \AA_{<d} \cap \AA_+$, and $\AA_{<d}^2 := \{ (a, b): a, b \in \AA_{<d} \}$.

We note that $n \leq p^{\valu_p (k-n)}$ if and only if $(T-1)^n \mid (T^{k-n} - 1)$. Let 
\[
 F(T) = \sum_{i = 0}^{k - 2n} \aalpha_i T^i
\]
be defined by $T^{k-n} - 1 = (T - 1)^n F(T)$. If $k = 2n$, then $n = k-n$ is a $p^\text{th}$-power and $F(T) = 1$. In general, by setting $T = 0$, we see that $\aalpha_0 = (-1)^{n+1}$.

\begin{lemma} \label{keylemma} Given $r > 0$, there exists a positive integer $d_r$ such that for all $d \geq d_r$ we have
\[
  \sum_{a \in \AA_{<d}} a^j = 0, \qquad \qquad \forall j,\  1 \leq j \leq r.
\]
\end{lemma}

\begin{proof} Define
\[
 S_{r, d}  := \sum_{a \in \AA_{<d}} a^r.
\]
If $(q-1) \nmid r$, then $S_{r, d} = 0$. This follows since 
\[
\AA_{<d} - \{ 0 \} = \{ \theta a_+: \theta \in \Fq^\ast, a_+ \in \AA_{<d+} \}
\] and summing over $\Fq^\ast$ first, we get $0$.

The result for $r \equiv 0 \mod (q-1)$ is due to Lee (see \cite[Section 5.6]{Tha}).
\end{proof}

\begin{rmk}
If $q = p^e$, then it is a result due to Lee that $S_{r, d} = 0$ whenever the sum of the $p$-adic digits of $r$ is  $< d e (p-1)$. A complete vanishing criterion was given by Carlitz. However, Carlitz simply asserts the result without proving it. It turns that the proof is not trivial and was only achieved by Sheats in the late 1990s. For more on this, see \cite[Sec.~5.6-5.8]{Tha} and the references therein.  
\end{rmk}

\begin{rmk}
 The previous lemma also follows easily from the vanishing of the Carlitz zeta function at negative `even' integers, which was first proved by Goss (see \cite[Sec.~8.8,~8.13]{GossBook}).
\end{rmk}

\begin{lemma} \label{lem2}
 If $d \geq d_{k-2n}$, then
\[
  \sideset{}{'} \sum_{(u, v) \in \AA_{<d}^2} \frac{(vz)^{k-n} - u^{k-n}}{(vz - u)^n} = \begin{cases} 0 & \quad k-2n \neq 0, \\
                                                                          -1  & \quad k-2n = 0.
                                                                         \end{cases}
\]
Here the prime on the summation means that we are taking pairs $(u, v) \neq (0, 0)$.
\end{lemma}

\begin{proof}
We break the sum into three parts. 

When $u = 0$ we have $v \neq 0$. By Lemma \ref{keylemma}, 
\[
 \sideset{}{'} \sum_{v \in \AA_{<d}} v^{k - 2n} z^{k-2n} = 0 \qquad \qquad \text{ for $k-2n \neq 0$}.
\]
The case $k - 2n = 0$ gives $-1$ for the sum, since we are summing over non-zero~$v$. Therefore,
\[
 \sideset{}{'} \sum_{v \in \AA_{<d}} v^{k - 2n} z^{k-2n} = \begin{cases} 0 & \quad k - 2n \neq 0, \\
                                             -1 & \quad k-2n = 0.
                                            \end{cases}
\]
When $v = 0$ we have $u \neq 0$. By the same argument as for the previous sum,
\[
 \sideset{}{'} \sum_{u \in \AA_{<d}} u^{k-2n} = \begin{cases} 0 & \quad k-2n \neq 0, \\
                                   -1 & \quad k-2n = 0.
                                  \end{cases}
\]
If $v \neq 0, u \neq 0$, then
\[
\begin{aligned}
  \sideset{}{'} \sum_{ u \in \AA_{<d}} \sideset{}{'} \sum_{v \in \AA_{<d}} \frac{(vz)^{k-n} - u^{k-n}}{(vz - u)^n} & = \sideset{}{'} \sum_{ u \in \AA_{<d}} \sideset{}{'} \sum_{v \in \AA_{<d}} u^{k-2n} F\left ( \frac{vz}{u} \right) \\
& = \sideset{}{'} \sum_{ u \in \AA_{<d}} \sideset{}{'} \sum_{v \in \AA_{<d}} \sum_{i = 0}^{k-2n} \aalpha_i (vz)^i u^{k-2n-i}.
\end{aligned}
\]
Summing over $v$ and using Lemma \ref{keylemma} we see that only the term $i = 0$ remains. But if $k-2n \neq 0$,  then for $i = 0$ we can sum over $u$ and get $0$. Therefore, if $k - 2n \neq 0$, then
\[
  \sideset{}{'} \sum_{ u \in \AA_{<d}} \sideset{}{'} \sum_{v \in \AA_{<d}} \frac{(vz)^{k-n} - u^{k-n}}{(vz-u)^n} = 0.
\]
On the other hand, if $k-2n = 0$, we have
\[
   \sideset{}{'} \sum_{ u \in \AA_{<d}} \sideset{}{'} \sum_{v \in \AA_{<d}} \frac{(vz)^{k-n} - u^{k-n}}{(vz - u)^n} = \sideset{}{'} \sum_{u \in \AA_{<d}} \sideset{}{'} \sum_{v \in \AA_{<d}} \aalpha_0 (vz)^0 u^0 = \aalpha_0 = (-1)^{n+1}. 
\]
Combining these proves the lemma.
\end{proof}

\begin{lemma} \label{lem3} Let $k - 2n > 0$. 
 If $d \geq d_{k-2 n}$, then for any $a, b \in T^d \AA$ (not both zero) we have
\[
 \sum_{(u, v) \in \AA_{<d}^2} \frac{(a+u)^{k-n}}{((a+u)z + b + v)^n} = \sideset{}{'} \sum_{(u, v) \in \AA_{<d}^2}  \frac{(bu - av)^{k-n}}{(az + b)^{k-n} ((a+u)z + b+v)^n}.
\]
Note that the left sum does not have the condition that $(u, v) \neq (0, 0)$.
\end{lemma}

\begin{proof}
Assume that $(u, v) \neq (0, 0)$. Then
\[
 \frac{(bu - av)^{k-n}}{(az + b)^{k-n} ((a+u)z + b+v)^n} - \frac{(a+u)^{k-n}}{((a+u)z + b + v)^n} 
\]
is equal to
\[
\frac{((bu-av) - (a+u)(az+b))^n \sum_{i = 0}^{k -2n} \aalpha_i (bu-av)^i ((a+u)(az+b))^{k-2n-i}}{(az + b)^{k-n}((a+u)z + b+v)^n}. 
\]
Here we have used the identity $X^{k-n} - Y^{k-n} = Y^{k-2n} (X-Y)^n F(X/Y)$. 
Since $(bu-av) - (a + u)(az + b) = -a ((a+u)z + b+v)$, the last expression reduces to
\[
  \frac{(-a)^n \sum_{i = 0}^{k -2n} \aalpha_i (bu-av)^i ((a+u)(az+b))^{k-2n-i}}{(az + b)^{k-n}}.
\]
For $1 \leq i \leq k-2n$, we consider
\[
  \sideset{}{'} \sum_{(u, v) \in \AA_{<d}^2} (bu - av)^i (a+u)^{k - 2n -i}.
\]
Expanding $(bu - av)^i$ by the binomial theorem and summing over $v$, we see by Lemma \ref{keylemma} that only the term $(bu)^i (a + u)^{k-2n-i}$ remains. Thus
\[
 \sideset{}{'} \sum_{(u, v) \in \AA_{<d}^2} (bu - av)^i (a+u)^{k - 2n -i} = \sideset{}{'} \sum_{u \in \AA_{<d}} (bu)^i (a + u)^{k-2n-i} .
\]
Expanding $(a + u)^{k -2 n - i}$ by the binomial theorem and summing over $u$, we obtain $0$ by Lemma \ref{keylemma}. 

For $i = 0$ we have
\[
 \begin{aligned}
 \sideset{}{'} \sum_{(u, v) \in \AA_{<d}^2} (a + u)^{k-2n} & = \sum_{u \neq 0} (a + u)^{k-2n} \sum_{v \in \AA_{<d}} 1 + a^{k-2n} \sum_{v \neq 0} 1 \\
 & = -a^{k-2n}.
\end{aligned}
\]
Therefore,
\[
  \sideset{}{'} \sum_{(u, v) \in \AA_{<d}^2} \frac{(-a)^n \sum_{i = 0}^{k -2n} \aalpha_i (bu-av)^i ((a+u)(az+b))^{k-2n-i}}{(az + b)^{k-n}}
\]
equals
\[
  \frac{(-1)^{n+1} a^{k-n}}{(az + b)^n} \aalpha_0.
\]
But $\aalpha_0 = (-1)^{n+1}$ (we are using $k \neq 2n$ here) and the lemma follows.
\end{proof}

\emph{Proof of Theorem~\ref{mainthm}.} In order to simplify notation, we impose the following conventions: in what follows we will assume that $a \in T^d \AA$, $b \in T^d \AA$, $u \in \AA_{<d}$ and $v \in \AA_{<d}$.

Define
\[
 \phi_{k, n} (z) := \sideset{}{'} \sum_{(u, v) } \frac{u^{k-n}}{(uz + v)^n} + \sideset{}{'} \sum_{(a, b) } \sideset{}{'} \sum_{(u, v) } \frac{(bu-av)^{k-n}}{(az + b)^{k-n} ((a+u)z + b + v)^n}.  
\]
Note that the first sum is finite, while the second sum converges, since each term is bounded by $\frac{1}{\mini \{ |a|, |b| \}^n}$ in absolute value.
We compute
\[
\begin{aligned}
 \phi_{k, n} \left ( \frac{-1}{z} \right) & =  \sideset{}{'} \sum_{(u, v)} \frac{u^{k-n} z^n}{(vz - u)^n} + \sideset{}{'} \sum_{(a, b)} \sideset{}{'} \sum_{(u, v)} \frac{z^k(bu-av)^{k-n}}{(bz-a)^{k-n} ((b + v)z - (a + u))^n} \\
\end{aligned}
\]
which by Lemma~\ref{lem2} equals
\[
\begin{aligned}
\qquad \quad &  \sideset{}{'} \sum_{(u, v)} \frac{(vz)^{k-n} z^n}{(vz - u)^n} + \sideset{}{'} \sum_{(a, b)} \sideset{}{'} \sum_{(u, v) } \frac{z^k(bu-av)^{k-n}}{(bz-a)^{k-n} ((b + v)z - (a + u))^n}. \\
\end{aligned}
\]
By replacing $u$ with $-u$ and $a$ with $-a$, we have
\[
\begin{aligned}
\qquad \quad & z^k \sideset{}{'} \sum_{(u, v)} \frac{v^{k-n}}{(vz + u)^n} + z^k \sideset{}{'} \sum_{(a, b) } \sideset{}{'} \sum_{(u, v)} \frac{(av-bu)^{k-n}}{(bz+a)^{k-n} ((b + v)z +(a + u))^n} \\
& = z^k \phi_{k, n}(z).
\end{aligned}
\]
Therefore, we have the correct functional equation with respect to $z \mapsto -1/z$.

It remains to show that $\phi_{k, n}$ has an $A$-expansion. By Lemma~\ref{lem3}
\[
 \sideset{}{'} \sum_{(a, b)} \sideset{}{'} \sum_{(u, v)} \frac{(bu-av)^{k-n}}{(bz-a)^{k-n} ((b + v)z - (a + u))^n}  = \sideset{}{
'}\sum_{(a, b)} \sum_{(u, v)} \frac{(a + u)^{k-n}}{((a+u) z + b + v)^n}. 
\]
Thus the sum defining $\phi_{k, n}$ is equal to
\[
\begin{aligned}
 & \sideset{}{'} \sum_{(u, v)} \frac{u^{k-n}}{(uz + v)^n} + \sideset{}{'} \sum_{b } \sum_{(u, v)} \frac{(bu)^{k-n}}{b^{k-n}(uz + b + v)^n} \\
  & \qquad \qquad \qquad \ \ \ \  + \sideset{}{'} \sum_{a} \sum_{b} \sum_{(u, v)} \frac{(a+u)^{k-n}}{(az + b)^{k-n} ((a+u)z + b + v)^n},
\end{aligned}
\]
which, after multiplying by $1/\ppi^{k-n}$, becomes
\[
 \sum_{u \in \AA_{<d}} u^{k-n} G_n(t_u) + \sum_{a \in T^d \AA} a^{k-n} G_n(t_a).
\]
Finally, notice that $G_n(t_{\theta a}) = \theta^{-n} G_n(t_a)$ and hence the expression above is precisely
\[
  - \sum_{a \in \AA_+} a^{k-n} G_n(t_a).
\]
This shows that $\phi_{k, n}$ is invariant under translations by $\AA$ (i.e., invariant under $z \mapsto z+a$ for all $a \in \AA$) and that 
\[
  \frac{-1}{\ppi^{k-n}} \phi_{k, n} = f_{k, n} = \sum_{a \in A_+} a^{k-n} G_n(t_a) \in S_{k, n}(\GL). \qquad \qquad \qquad ~\square
\]

\begin{rmk}
We want to briefly mention two cases outside of Theorem~\ref{mainthm} which are of interest.

First, if $k = n$, we can make the same definition for $\phi_{k, n}$, i.e.,
\[
  \phi_{k, n} (z) = \sideset{}{'} \sum_{(u, v) } \frac{1}{(uz + v)^n} + \sideset{}{'} \sum_{(a, b) } \sideset{}{'} \sum_{(u, v) } \frac{1}{ ((a+u)z + b + v)^n}.  
\]
Then $\phi_{k, n}$ has the correct functional equation under $z \mapsto -1/z$, but Lemma \ref{lem3} does not apply, since
\[
 \sideset{}{'} \sum_{(a, b) } \sideset{}{'} \sum_{(u, v) } \frac{1}{ ((a+u)z + b + v)^n} \neq  \sum_{(a, b) } \sideset{}{'} \sum_{(u, v) } \frac{1}{ ((a+u)z + b + v)^n}.  
\]
Therefore, $\phi_{k, n}$ does not have a $t$-expansion. To fix this we add the $(u, v) = (0, 0)$ term to the double sum. The resulting expression is essentially the non-normalized Eisenstein series $E_n$: 
\[
   \phi_{k, n} + \sideset{}{'} \sum_{(a, b)} \frac{1}{ (az + b)^n} = \ppi^n E_n.
\]

\

The second case is $k = 2n$. As $n \leq p^{\valu_p(k-n)}$, we see that in this case $k-n = n = p^\nu$ for some non-negative integer $\nu$. We define $\phi_{k, n}$ as in the proof. Using Lemma~\ref{lem2}, we obtain
\[
  \phi_{k, n} \left ( \frac{-1}{z} \right) = z^k \phi_{k, n} + z^n.
\]
Lemma~\ref{lem3} does not hold, but it is replaced by the equation
\[
 -\sum_{(u, v) \in \AA_{<d}^2} \frac{(a+u)^{k-n}}{((a+u)z + b + v)^n} = \sideset{}{'} \sum_{(u, v) \in \AA_{<d}^2}  \frac{(bu - av)^{k-n}}{(az + b)^{k-n} ((a+u)z + b+v)^n}.
\]
Therefore, if we define
\[
 \phi_{k, n}^\ast  := \sideset{}{'} \sum_{(u, v) } \frac{u^{k-n}}{(uz + v)^n} - \sideset{}{'} \sum_{(a, b) } \sideset{}{'} \sum_{(u, v) } \frac{(bu-av)^{k-n}}{ (az + b)^{k-n} ((a+u)z + b + v)^n},
\]
we have 
\[
   \phi_{k, n}^\ast \left ( \frac{-1}{z} \right ) = z^k \phi_{k, n}^\ast + z^n
\]
and
\[
 \frac{-1}{\ppi^{k-n}} \phi_{k, n}^\ast = f_{k, n}^\ast =  \sum_{a \in \AA_+} a^{k-n} G_n (t_a) = \sum_{a \in \AA_+} a^{p^\nu} G_{p^\nu}(t_a). 
\]
The first equation resembles the functional equation of a Drinfeld quasi-modular form (see \cite[Def.~2.1]{BoPe}) and the second equation shows that $f_{k, n}^\ast$ is $E^{p^\nu}$, the $p^\nu$-th power of the false Eisenstein series (\cite[(8.2)]{Gek}):
\[
	E  := \sum_{a \in \AA_+} a t_a.
\]
\end{rmk}

\section{Acknowledgements}

The present work grew out of the author's Ph.D. thesis \cite{thesis}. The author is greatly indebted to Dinesh Thakur for suggesting the original project that became the dissertation. His support, patience, and insistence on clarity and precision have been instrumental for the successful completion of the present work. The author wants to thank Bartolom\'{e} L\'{o}pez, Gebhard B\"{o}ckle, David Goss, Ernst-Ulrich Gekeler, Matthew Johnson, Frederico Pellarin and Romyar Sharifi, for making numerous suggestions and comments about previous version of the present work. The author also wants to acknowledge the support of the University of Arizona during the writing of \cite{thesis}.

\bibliographystyle{amsplain}
\bibliography{art}

\providecommand{\bysame}{\leavevmode\hbox to3em{\hrulefill}\thinspace}
\providecommand{\MR}{\relax\ifhmode\unskip\space\fi MR }
\providecommand{\MRhref}[2]{%
  \href{http://www.ams.org/mathscinet-getitem?mr=#1}{#2}
}
\providecommand{\href}[2]{#2}
\begin{thebibliography}{10}

\bibitem{Cec}
C\'{e}cile Armana, \emph{Coefficients of {D}rinfeld modular forms and {H}ecke
  operators}, J.~Number Theory \textbf{131} (2011), no.~8, 1435--1460.

\bibitem{Boe1}
Gebhard B\"{o}ckle, \emph{An {E}ichler-{S}himura isomorphism over function
  fields between {D}rinfeld modular forms and cohomology classes of crystals},
  preprint, 2002.

\bibitem{BP}
Gebhard B\"{o}ckle and Richard Pink, \emph{Cohomology theory of crystals over
  function fields}, EMS Tracts in Mathematics, vol.~9, European Mathematical
  Society, 2009.

\bibitem{BoPe}
Vincent Bosser and Frederico Pellarin, \emph{Hyperdifferential properties of
  {D}rinfeld quasi-modular forms}, Int. Math. Res. Not. 2008 (2008).

\bibitem{Gek1}
Ernst-Ulrich Gekeler, \emph{A product expansion for the discriminant function
  of {D}rinfeld modules of rank two}, J. Number Theory \textbf{21} (1985),
  135--140.

\bibitem{Gek}
\bysame, \emph{On the coefficients of {D}rinfeld modular forms}, Invent. Math.
  \textbf{93} (1988), 667--700.

\bibitem{Go1}
David Goss, \emph{Modular forms for $\mathbb{F}_r[t]$}, J. Reine Angew. Math.
  \textbf{317} (1980), 16--39.

\bibitem{Go2}
\bysame, \emph{$\pi$-adic {E}isenstein series for function fields}, Compos.
  Math. \textbf{41} (1980), 3--38.

\bibitem{GossBook}
\bysame, \emph{Basic structures of function field arithmetic}, Ergebnisse der
  Mathematik und ihrer Grenzgebiete, vol.~35, Springer, 1998.

\bibitem{Jo1}
Matthew Johnson, \emph{Hecke eigenforms as products of eigenforms}, J.~Number
  Theory \textbf{133} (2013), 2339--2362.

\bibitem{Lop1}
Bartolom\'{e} L\'{o}pez, \emph{A non-standard {F}ourier expansion for the
  {D}rinfeld discriminant function}, Archiv der Mathematik \textbf{95} (2010),
  143--150.

\bibitem{Lop2}
\bysame, \emph{Action of {H}ecke operators on two distinguished {D}rinfeld
  modular forms}, Archiv der Mathematik \textbf{97} (2011), 423--429.

\bibitem{Pel1}
Frederico Pellarin, \emph{Estimating the order of vanishing at infinity of
  {D}rinfeld quasi-modular forms}, arXiv: 0907.4507v3, 2010.

\bibitem{Pel2}
\bysame, \emph{$\tau$-recurrent sequences and modular forms}, arXiv:
  1105.5819v3, 2011.

\bibitem{thesis}
Aleksandar Petrov, \emph{On {A}-expansions of {D}rinfeld modular forms}, Ph.D.
  thesis, The University of Arizona, 2012.

\bibitem{Tha}
Dinesh Thakur, \emph{Function field arithmetic}, World Scientific, 2004.

\bibitem{Vincent}
Christelle Vincent, \emph{Drinfeld modular forms modulo $\ppp$}, Proc. Amer.
  Math. Soc. \textbf{138} (2010), 4217--4229.

\end{thebibliography}

\end{document}